\documentclass[11pt,a4paper]{article}
\usepackage{a4wide,enumitem}
\usepackage{amsmath,amssymb,amsthm,bm}
\usepackage[english]{babel}
\usepackage[utf8]{inputenc}
\usepackage{t1enc, lmodern}
\usepackage{color}
\usepackage[nonamebreak, round]{natbib}
\usepackage[bookmarks=false,colorlinks=false,linkcolor=blue]{hyperref}
\usepackage{ifpdf}

\ifpdf
\usepackage[pdftex]{graphicx} 
\DeclareGraphicsExtensions{.pdf} 
\else
\usepackage[dvips]{graphicx}  
\DeclareGraphicsExtensions{.eps} 
\fi

\def\E{{\mathbb E}}

\def\P{{\mathbb P}}

\def\R{{\mathbb R}}
\def \L{\mathbb{L}}

\def\bb{\bm{b}}
\def\be{\bm{e}}

\def\bF{\mathbb{F}}

\def\bmu{\bar \mu}

\def \cf{\mathcal{F}}

\def \cA{\mathcal{A}}

\def\expp#1{\mathop {\mathrm{e}^{ #1}}}

\def\norm#1{\left|#1\right|}
\def\inv#1{\frac{1}{#1}}
\def\ind#1{\; {\mathbf 1}_{\{#1\}}}

\numberwithin{equation}{section}

\def\sphere{\mathcal{S}(\R^3)}

\theoremstyle{plain}
\newtheorem{thm}{Theorem}
\newtheorem{prop}[thm]{Proposition}
\newtheorem{lemma}[thm]{Lemma}
\newtheorem{corollary}[thm]{Corollary}
\newtheorem{remark}[thm]{Remark}

\emergencystretch=\hsize
\begin{document}

\title{Long time behaviour of a stochastic nano particle}

\author{Pierre \'Etor\'e\thanks{Grenoble INP, Laboratoire Jean Kuntzmann, 51 rue des
Mathématiques, 38041 Grenoble cedex 9, France; pierre.etore@imag.fr} \and St\'ephane
Labb\'e\thanks{Université Grenoble 1, Laboratoire Jean Kuntzmann, 51 rue des
Mathématiques, 38041 Grenoble cedex 9, France; stephane.labbe@imag.fr (Grant
HM-MAG, RTRA Foundation) } \and J\'er\^ome
Lelong\thanks{Grenoble INP, Laboratoire Jean Kuntzmann, 51 rue des
Mathématiques, 38041 Grenoble cedex 9, France; jerome.lelong@imag.fr}\\
}

\date{\today}

\maketitle

\begin{abstract}  
  In this article, we are interested in the behaviour of a single ferromagnetic
  mono-domain particle submitted to an external field with a stochastic
  perturbation. This model is the first step toward the mathematical
  understanding of thermal effects on a ferromagnet.  In a first part, we
  present the stochastic model and prove that the associated stochastic
  differential equation is well defined. The second part is dedicated to the
  study of the long time  behaviour of the magnetic moment and in the third part
  we prove that the stochastic perturbation induces a non reversibility
  phenomenon.  Last, we illustrate these results through numerical simulations
  of our stochastic model.\\ The main results presented in this article are on
  the one hand the rate of convergence of the magnetization toward the unique
  stable equilibrium of the deterministic model and on the other  hand a sharp
  estimate of the hysteresis phenomenon induced by the stochastic perturbation
  (remember that with no perturbation, the magnetic moment remains constant).
  \\


\end{abstract}

\section{Introduction} Thermal effects in ferromagnetic materials are essential
in order to understand their behaviour at ambient temperature or, more
critically, in electronic devices where the Joule effect induces high heat
fluxes. This effect is commonly modeled by the introduction of a noise at
microscopic scale on the magnetic moment direction and at the mesoscopic scale
by a transition of behaviour. In ferromagnetic materials, the transition between
the non-linear behaviour and the linear behaviour is managed by the struggle
between the Heisenberg interaction and the disorder induced by the heating.
This model explains the critical temperatures such as the Curie temperature for
ferromagnetic materials. In this context, it is essential to understand the
impact of introducing stochastic perturbations in deterministic models of
ferromagnetic materials such as the micromagnetism (see~\cite{Brown:magnet,Brown:Microm}).\\

The understanding of this phenomena is a key point in order to simulate
realistic ferromagnetic devices such as micro electronic circuits.
Furthermore, heating has a real effect on the microstructure
dynamics in magnets; then, efficiently controlled, the dynamics of microstructures
could accelerate processes such as the magnetization switching, which is the
basics of magnetic recording techniques.\\

During the last decade, several studies have been initialized in several
articles by physicists (e.g.
\cite{Mercer:Atomic,Zheng:Atomistic,Raikher:Dynamic,Atkison:Magnetic,Raikher:Magnetization,Scholz:Micromagnetic,Martinez:Micromagnetic,Neil:Modeling}),
but to our knowledge very few if no mathematical models justifying this kind of
effects at the micro-scale have been developed for continuous state
magnetic spins. Numerous studies focused on the discrete state spin approach
(in particular the Glauber dynamic model, see for instance \cite{MR2642889}),
but in our context, we try to catch the magnetization defects induced by thermal
effects in ferromagnetic materials seen from the dynamical point of view of the
Larmor precession equation.  In this article, our goal is to improve the
understanding of thermal effects in ferromagnets. To achieve this goal, we focus
this first study on the dynamic of a single magnetic moment submitted to a
stochastic perturbation.  Our main aim is to characterize precisely the dynamic
of the moment, giving estimates and general behaviour in long time. In
particular, we will exhibit an hysteresis behaviour of the magnetization in our
model.\\

The model we are studying mimics the behaviour of a single magnetic moment
$\mu(t)$ (function from $\mathbb{R}$ into
$S(\mathbb{R}^3)=\{u\in\mathbb{R}^3;\;\vert \mu\vert=1\}$) submitted to an
external field $b$. The dynamic of such a system is, at the micro-scale,
described by the Larmor precession equation 
\begin{eqnarray*}
  \frac{d\mu}{dt}=-\mu\wedge b.
\end{eqnarray*}
Nevertheless, this equation is non dissipative and, in order to make the
theoretical study easier, we introduce a dissipative part using the Landau-Lifchitz
equation
\begin{eqnarray*}
  \frac{d\mu}{dt}=-\mu\wedge b-\alpha \mu\wedge (\mu\wedge b),
\end{eqnarray*}
where $\alpha$ is a positive real constant and we set the initial condition
$\mu(0)=\mu_0 \in S(\mathbb{R}^3)$.  We point out two major properties of this
system
\begin{enumerate}[label=\textit{\roman{enumi}.}]
  \item {$\forall t\in\mathbb{R}, \vert\mu(t)\vert=1$,}
  \item {$\displaystyle\forall t\in\mathbb{R},\; \frac{d}{dt}(\mu(t)\cdot b)\ge 0$.}
\end{enumerate}
The first property, which is definitely essential, will have to be preserved by
the stochastic system and the second property is the energy decreasing induced
by the introduction of the dissipation term. The dynamic of this deterministic
system is classical; in fact, one knows that $\displaystyle
\lim_{t\rightarrow\infty}\mu(t)=\frac{1}{\vert b\vert}b$ provided that $\mu(0)
\ne -b$.  In this work, we will develop such a result for the stochastic system.
In order to build this stochastic system, the first question is how to introduce
the stochastic perturbation in the deterministic system. We want to model the
thermal effects which are external perturbations of the magnetic moment. In
fact, this perturbation could be modeled has an external perturbation field. In
the sequel, we will choose to build a stochastic system by perturbing the
external field with a Brownian motion. We will write down 
\begin{equation*}
  \begin{cases}
    dY_t &=  - \mu_t \wedge (b\;  dt + \varepsilon\; dW_t) - \alpha \mu_t \wedge
    (\mu_t \wedge (b \; dt + \varepsilon \; dW_t) ) \\
    Y_0&=y\in \sphere.
  \end{cases}
\end{equation*}
where $\varepsilon$ is a strictly positive real number and $W$ a standard
Brownian motion with values in $\R^3$. But, an easy
computation of $d(|Y_t|^2)$ using It\^o's formula shows that 
the process $Y$ will not stay in $S(\mathbb{R}^3)$, then, in order to preserve
this essential behaviour, we have to renormalise the previous equation and set
$$
\mu_t  = \frac{Y_t} {|Y_t|}.
$$
Given this system, we prove the following results 
\begin{enumerate}[label=\textit{\roman{enumi}.}]
  \item $\displaystyle \mu_t\cdot b \xrightarrow[t \rightarrow \infty]{}
    \vert b\vert$, a.s.,

  \item  $\lim_{t \longrightarrow \infty} \sqrt{t} \; \E[ \norm{\norm{b} - \mu_t
    \cdot b} ]$ exists and we can compute its value.
\end{enumerate}

Note that these results are only valid for $\alpha>0$, as for $\alpha=0$, it is
easy to show that the function $e(t) = \E(\mu_t \cdot b)$ satisfies the
following ordinary differential equation $e'(t) = -e(t) \frac{h'(t)}{h(t)}$.
Hence, $e(t) = \frac{e(0)}{h(t)} \xrightarrow[t \to \infty]{} 0$. This
contradicts the a.s. convergence of $\mu_t$ to $\frac{b}{\norm{b}}$.

\begin{enumerate}[label=\textit{\roman{enumi}.},resume]
  \item When $\mu$ is submitted to a time varying external field, an hysteresis
    phenomenon appears. If we consider $\bb \in \sphere$ and let $b$ linearly vary
    between $+\bb$ and $-\bb$ over the time interval $[0,T]$, then 
    $\E(\mu_t \cdot \bb)$ is bounded from below by $\frac{1}{\sqrt{1+c t}}$ for
    $t \le T/2$ where $c$ is a constant depending only on $\varepsilon$ and
    $\alpha$.
\end{enumerate}
First, we make precise the derivation of the stochastic model and discuss the
physical differences between the Itô interpretation or the Stratonovich one of
the noise term. Then, we lead a detailed study of its asymptotic behaviour and
in particular we point out an hysteresis phenomenon. This phenomenon is obtained
by slow variations of the external field such that the dynamic of relaxation of
the magnetization toward this field becomes instantaneous when the speed ratio
of the external excitation goes to zero. The results shown in this article are
finally illustrated by numerical simulations.

\paragraph{Notations:}

\begin{itemize}
  \item For $a$ and $b$ in $\R^3$ we denote by $a\cdot b$ their
    scalar product,  $a\cdot b=\sum_{i=1}^3 a^ib^i$. 
  \item For $a$ in $\R^3$, we denote by $|a| = \sqrt{a \cdot a}$ the Euclidean
    norm of $a$.
  \item We like to encode elements of $\R^3$ as column vectors. For $x \in \R^3$,
    $x^*$ is a row vector. Similarly, we use the star notation ``${}^*$'' to
    denote the transpose of matrices.
  \item If $H=(H_t)_{t\geq 0}$ is a $3$-dimensional $\bF-$adapted process
    satisfying $\int_0^t |H_u|^2 du < \infty$ a.s. for all $t$,  we may write
    $\int_0^t H_u\cdot dW_u$ for $\sum_{i=1}^3 \int_0^t H^i_u\,dW^i_u$ and use the
    differential form $H_t\cdot dW_t$ for $\sum_{i=1}^3 H^i_t\,dW^i_t$.
\end{itemize}

\section{Mathematical model}

\subsection{First properties of the model}

Let $(\Omega, \cA, \P)$ be a probability space.  We consider a  standard Brownian
motion $W$ defined on this space with values in $\R^3$ and denote by $\bF =
(\cf_t)_{t \ge 0}$ its natural filtration augmented with $\P$null sets.

Let $b\in\R^3$ be the magnetic field. We model the $\sphere$-valued magnetic
moment process $\mu=(\mu_t)_{t\geq 0}$ by the following coupled stochastic
differential equation (SDE in short)
\begin{equation}
  \label{eq:sys_sto}
  \begin{cases}
    dY_t &=  - \mu_t \wedge (b\;  dt + \varepsilon\; dW_t) - \alpha \mu_t \wedge
    (\mu_t \wedge (b \; dt + \varepsilon \; dW_t))  \\
    \mu_t & = \frac{Y_t} {|Y_t|}\\
    Y_0&=y\in \sphere,
  \end{cases}
\end{equation}
where $\alpha >0$ is the magnitude of the damping term and $\varepsilon>0$ is
the magnitude of the noise term.

The term $\mu_t\wedge dW_t$ in \eqref{eq:sys_sto} is naturally defined by
introducing the antisymmetric operator $L : \R^3 \longmapsto \R^{3 \times 3}$
associated to the vector product in $\R^3$
\begin{equation*}
  L(x) = \left(
  \begin{array}{ccc}
    0 & -x^3 & x^2\\
    x^3& 0&-x^1 \\
    -x^2 & x^1 & 0\\
  \end{array}
  \right).
\end{equation*}
Hence, for a $3$-dimensional $\bF-$adapted process $H$ satisfying $\int_0^t
\norm{H_s}^2 ds<\infty$ a.s. for all $t>0$, the process $\int_0^t H_s\wedge
dW_s$ is defined by $\int_0^t L(H_s) dW_s$, which is a standard
multi--dimensional It\^o stochastic integral. When dealing with the differential
expression, we will either write $H_t \wedge dW_t$ or $L(H_t) dW_t$.

\begin{prop}\label{prop:norm_Y}
  Let $(Y, \mu)$ be a pair of processes satisfying~\eqref{eq:sys_sto}, then
  \begin{equation*}
    d\norm{Y_t}^2=2\varepsilon^2 (\alpha^2 + 1) dt
  \end{equation*}
  and therefore $\norm{Y_t} = \sqrt{2 \varepsilon^2 (\alpha^2 + 1)t + 1}$ is non random.
\end{prop}

\begin{remark}
  The fact that $\norm{Y_t}$ is non random is definitely essential in all the
  following computations. In particular, we deduce from this result that
  $\norm{Y_t}$ has finite variation.
\end{remark}

\begin{proof}
  Using It\^o's lemma we have
  $$
  d\norm{Y_t}^2=2Y_t\cdot dY_t+\sum_{i=1}^3d\langle Y^i,Y^i\rangle_t=\sum_{i=1}^3d\langle Y^i,Y^i\rangle_t,
  $$
  where we have used the fact that $Y_t$ and $dY_t$ are orthogonal. 
  But, using the identity $a\wedge(b\wedge c)=(a\cdot c)b-(a\cdot b)c$, we have
  \begin{align*}
    dY_t&=\varepsilon\big[  -(\mu_t\wedge dW_t) - \alpha ( (\mu_t\cdot
    dW_t)\mu_t-(\mu_t\cdot\mu_t)dW_t) \big]+\ldots  dt\\
    &=\varepsilon  A(\mu_t)\,dW_t+\ldots dt,
  \end{align*}
  where we have set $A(\mu)=\alpha I-\alpha(\mu\mu^*)-L(\mu)$ and used $|\mu_t|=1$.
  Thus, 
  \begin{align*}
    d\langle Y,Y\rangle_t&=    \varepsilon^2  AA^*(\mu_t)dt \\
    &=\varepsilon^2\Big[ \alpha^2I-\alpha^2(\mu_t\mu_t^*)+\alpha L(\mu_t) -\alpha^2(\mu_t\mu_t^*)+\alpha^2(\mu_t\mu_t^*)(\mu_t\mu_t^*) -\alpha(\mu_t\mu_t^*)L(\mu_t)  \\
    &\hspace{1cm}-\alpha L(\mu_t)+\alpha L(\mu_t)(\mu_t\mu_t^*)-L(\mu_t)L(\mu_t)\Big]\,dt\\
    &=\varepsilon^2\Big[  \alpha^2I-\alpha^2(\mu_t\mu_t^*)+L(\mu_t)L^*(\mu_t)  \Big]dt,
  \end{align*}
  where we have used $L(\mu)\mu=0$, $L^*(\mu)=-L(\mu)$ and again $\mu_t^*\mu_t=1$.
  Thus, for each $1\leq i\leq 3$ we have
  $$
  d\langle Y^i,Y^i\rangle_t=\varepsilon^2\big[\alpha^2-\alpha^2(\mu^i_t)^2+\sum_{k=1}^3(L_{ik}(\mu_t))^2\big].$$

  Then, summing over $i$, we get
  \begin{align*}
    d\norm{Y_t}^2&=\varepsilon^2\big[3\alpha^2-\alpha^2\norm{\mu_t}^2 
    +\sum_{i,j=1}^3(L_{ij}(\mu_t))^2\Big]dt\\
    &= \varepsilon^2[3\alpha^2-\alpha^2\norm{\mu_t}^2+2\norm{\mu_t}^2]dt,
  \end{align*}
  The result ensues by remembering that $\norm{\mu_t}^2 = 1$.
\end{proof}

With the help of Proposition~\ref{prop:norm_Y}, we can establish the SDE
satisfied by the one dimensional process $(\mu_t \cdot b)_t$.  We introduce the
function 
\begin{equation}
  \label{eq:h}
  h(t) = \norm{Y(t)} = \sqrt{2 \varepsilon^2 (\alpha^2 + 1)t + 1}.
\end{equation}
Since $\norm{Y_t}$ is non random, we deduce from Equation~\eqref{eq:sys_sto}
that 
\begin{align}
  \label{eq:sde_mu}
  d\mu_t & = - \frac{\mu_t h'(t)}{h(t)} dt + \frac{dY_t}{h(t)} \nonumber \\
  & =  - \frac{\mu_t h'(t)}{h(t)} dt - \frac{1}{h(t)} \left( \mu_t\wedge (b\;
  dt + \varepsilon\; dW_t) + \alpha \mu_t \wedge (\mu_t
  \wedge (b \; dt + \varepsilon \; dW_t)) \right) \nonumber \\
  d\mu_t   & =  - \frac{\mu_t h'(t) + \mu_t \wedge
  b + \alpha ( \mu_t (\mu_t \cdot b) - b )}{h(t)}
  dt - \frac{\varepsilon}{h(t)} \left( L(\mu_t) + \alpha (\mu_t \mu_t^*
  -I)\right) dW_t 
\end{align}
By taking the scalar product with $b$, we get
\begin{align}
  \label{eq:sde_scalar}
  d( \mu_t \cdot b) & =  - (\mu_t \cdot b) \frac{h'(t)}{h(t)} dt -
  \frac{\alpha}{h(t)} \left( (\mu_t \cdot
  b)^2 - \norm{b}^2 \right) dt \nonumber \\ 
  & \qquad - \frac{\varepsilon}{h(t)} \left( (\mu_t
  \wedge  \; dW_t) \cdot b + \alpha (\mu_t \cdot b) (\mu_t \cdot dW_t) -
  \alpha (b \cdot dW_t) \right) \nonumber \\
  d( \mu_t \cdot b) & =  - (\mu_t \cdot b) \frac{h'(t)}{h(t)} dt -
  \frac{\alpha}{h(t)} \left( (\mu_t \cdot
  b)^2 - \norm{b}^2 \right) dt \nonumber \\ 
  & \qquad - \frac{\varepsilon}{h(t)} \big( -L(\mu_t) b
  + \alpha ((\mu_t \cdot b) \mu_t -b)  \big) \cdot dW_t
\end{align}
We may call this equation the SDE satisfied by $\mu_t \cdot b$ whereas it is not
an SDE properly speaking since the r.h.s member actually depends on all the
components of $\mu_t$ and not only on $\mu_t \cdot b$. Nonetheless, we may use
this abuse of terminology throughout the paper.
\begin{remark}[Remark on the existence and uniqueness of solutions to
  Equation~\eqref{eq:sys_sto}]
  Let us consider the following coupled SDE
  \begin{subequations}
    \label{eq:sys_sto_coupled}
    \begin{align}
      \label{eq:sys_sto_coupled_Y}
      dY_t &=  - \mu_t \wedge (b\;  dt + \varepsilon\; dW_t) - \alpha \mu_t \wedge
      (\mu_t \wedge (b \; dt + \varepsilon \; dW_t))  \\
      \label{eq:sys_sto_coupled_mu}
      d\mu_t   & =  - \frac{\mu_t h'(t) + \mu_t \wedge
      b + \alpha ( \mu_t (\mu_t \cdot b) - b )}{h(t)}
      dt - \frac{\varepsilon}{h(t)} \left( L(\mu_t) + \alpha (\mu_t \mu_t^*
      -I)\right) dW_t \\
      Y_ 0 & = \mu_ 0 \in \sphere \nonumber
    \end{align}
  \end{subequations}
  This system is actually decoupled as the SDE on $\mu$ is autonomous (this has
  only been possible because $\norm{Y_t}$ is non random). The existence and
  uniqueness of a solution to Equation~\eqref{eq:sys_sto_coupled} boil down to
  the ones of Equation~\eqref{eq:sys_sto_coupled_mu}. By computing
  $d(|\mu_t|^2)$, we deduce that if there exists a solution $\mu$ to
  Equation~\eqref{eq:sys_sto_coupled_mu}, $\norm{\mu_t}^2 = 1$ a.s. for all t.
  Hence, it is sufficient to check the standard global Lipschitz behaviour of
  the coefficients on $\R_+ \times \sphere$ to prove existence and uniqueness of
  a strong solution to Equation~\eqref{eq:sys_sto_coupled_mu}. 

  We have already seen above that if a pair $(Y, \mu)$ is solution of
  Equation~\eqref{eq:sys_sto}, it also solves
  Equation~\eqref{eq:sys_sto_coupled}. 

  Conversely, if $(Y, \mu)$ is the unique strong solution of
  Equation~\eqref{eq:sys_sto_coupled}, it is clear that $\norm{Y_t} = h(t)$ by
  following the proof of Proposition~\ref{prop:norm_Y} and moreover the
  computation of $d(Y_t / \norm{Y_t})$ shows that the process $(Y_t /
  \norm{Y_t})_t$ solves the same SDE as $\mu$, hence for all $t$ $\mu_t =
  \frac{Y_t}{\norm{Y_t}}$ a.s. This last argument proves that
  Equations~\eqref{eq:sys_sto} and~\eqref{eq:sys_sto_coupled} have the same
  solutions. Therefore, we deduce that Equation~\eqref{eq:sys_sto} admits a
  unique strong solution denoted by $(Y, \mu)$ in the sequel.
\end{remark}

\subsection{Why we did not choose a Stratonovich rule for the perturbation.}

One may argue that if we had written the stochastic perturbation in a
Stratonovich way, the norm of the stochastic magnetic moment would have remained
constant by construction. This may sound as a favorable argument to go the
Stratonovich way but while digging into the two approaches, it eventually
becomes clear that the Itô rule and the Stratonovich do not model the same
physical phenomenon. This investigation has led us to add a few results on the
Stratonovich approach.

\paragraph{The Stratonovich stochastic perturbation.}
Let $\partial$ denote the Stratonovich differential operator. In this paragraph, the
process $(\bmu_t)_t$ denotes the stochastic system with a Stratonovich perturbation.
\begin{align}
  \label{eq:sys_sto_strato}
  \partial \bmu_t &=  - \bmu_t \wedge (b\;  \partial t + \varepsilon\; \partial W_t) 
  - \alpha \bmu_t \wedge \bmu_t \wedge (b \; \partial t + \varepsilon \;
  \partial W_t), \quad \bmu_0=y\in \sphere.
\end{align}
The dynamics of $(\bmu_t)_t$ can be rewritten using the operator $A : \R^3
\longrightarrow \R^{3 \times 3}$ defined by $A(x)=\alpha I-\alpha x x^* - L(x)$
\begin{align}
  \label{eq:sys_sto_strato2}
  \partial \bmu_t &=  A(\bmu_t) b\; \partial t + \varepsilon A(\bmu_t) \partial W_t
\end{align}
Now, we turn this Stratonovich SDE into an Itô SDE (see~\cite[V.30]{rogers00})
\begin{align*}
  d \bmu_t &=   A(\bmu_t) b\; dt + \varepsilon A(\bmu_t) d W_t + \frac{1}{2}
  \varepsilon^2 \sum_{q=1}^3 \sum_{j=1}^3 (A_{jq}  D_j (A_{iq}))(\bmu_t),
\end{align*}
where $D_j$ denotes the partial derivative with respect to the $j-th$ component.

\begin{align*}
  \sum_{q=1}^3 \sum_{j=1}^3 (A_{jq}  D_j (A_{iq}))(x) =  \sum_{j=1}^3 (D_j A)
  A^*_{\cdot j} (x).
\end{align*}
Let us compute $D_j A$ by using the fact that $D_j(x) = \be_j$, where $\be_j$
is the $j-th$ vector of the canonical basis.
\begin{align*}
  D_j A (x) & = -\alpha (D_j x x^* + x D_jx^* ) - L(D_j (x)) \\
  & = -\alpha (\be_j  x^* + x \be_j^* ) - L(\be_j).
\end{align*}
Then, we get
\begin{align*}
  \sum_{j=1}^3 (D_j A) A^*_{\cdot j} (x) &= \sum_{j=1}^3 \left( -\alpha
  (\be_j  x^* + x \be_j^* ) - L(\be_j) \right) \left(\alpha \be_j -\alpha x x_j -
  L(e_j) x \right) \\
  &= \sum_{j=1}^3 \alpha^2 (-x_j \be_j  - x + x_j \be_j  + x_j^2 x) +
  L(\be_j) L(\be_j) x \\
  &= -2 \alpha^2 x + \sum_{j=1}^3 (\be_j \cdot x) \be_j - (\be_j \cdot \be_j ) x \\
  &= -2 (\alpha^2 + 1) x.
\end{align*}
Hence, the process $(\bmu_t)_t$ solves the following Itô SDE
\begin{align*}
  d \bmu_t &=   (A(\bmu_t) b  - \varepsilon^2 (\alpha^2 + 1) \bmu_t)  d t +
  \varepsilon A(\bmu_t) d W_t.
\end{align*}
We easily check that $\norm{\bmu_t}=1$, as expected.
From this equation, we compute 
\begin{align*}
  d( \bmu_t \cdot b) & =   \left\{ \alpha \left(-(\bmu_t \cdot b)^2 +
  \norm{b}^2 \right) - \varepsilon^2 (\alpha^2 +1) \bmu_t \cdot b \right\} dt
  \\
  & \quad - \varepsilon \big( -L(\bmu_t) b + \alpha ((\bmu_t \cdot b) \bmu_t -b)
  \big) \cdot dW_t.
\end{align*}
It is easy to check that 
\begin{equation*}
  (d(\bmu_t \cdot b))_{\Big| \bmu_t = b / \norm{b}} = - \varepsilon^2 (\alpha^2
  + 1) \norm{b} dt.
\end{equation*}
As $|\bmu_t| = 1$, the stochastic integral is a true martingale nd we easily get
\begin{align*}
  \E[\bmu_t \cdot b]' & = \alpha (|b|^2 - \E[(\bmu_t \cdot b)^2])  - \varepsilon^2
  (\alpha^2 +1) \E[\bmu_t \cdot b ]  \\
  \E[\bmu_t \cdot b] - \E[\bmu_0 \cdot b] \expp{-\varepsilon^2(\alpha^2+1) t} &=  
  \expp{-\varepsilon^2(\alpha^2+1) t}  \int_0^t \alpha (|b|^2 - \E[(\bmu_s \cdot b)^2]) 
  \expp{\varepsilon^2(\alpha^2+1) s} ds \\
  \limsup_{t \rightarrow +\infty} \E[\bmu_t \cdot b] & = \frac{|b|^2
  \alpha}{\varepsilon^2 (\alpha^2+1)}- \liminf_{t \rightarrow
  +\infty}  \expp{-\varepsilon^2(\alpha^2+1) t}  \int_0^t \alpha \E[(\bmu_s \cdot b)^2]) 
  \expp{\varepsilon^2(\alpha^2+1) s} ds  \\
  \limsup_{t \rightarrow +\infty} \E[\bmu_t \cdot b] & \le
  \frac{\alpha}{\varepsilon^2 (\alpha^2+1)} \left( |b|^2 
  - \liminf_{t \rightarrow +\infty}    \E[(\bmu_t \cdot b)^2] \right)
\end{align*}
where the last inequality comes from Lemma~\ref{lem:liminfexp}. This implies
that $b$ cannot be an equilibrium point of the stochastic system $(\bmu_t)_t$.
Moreover, if $\bmu_t \cdot b$ were converging to some deterministic value $l \in
[-|b|, |b|]$, we would get $l = \frac{\alpha}{\varepsilon^2(\alpha^2+1)} (|b|^2
- l^2)$. The only physically acceptable solution would be $l = \frac{1}{2} \left(
\sqrt{4 |b|^2 +\frac{\varepsilon^4(\alpha^2+1)^2}{\alpha^2}} -
\frac{\varepsilon^2(\alpha^2+1)}{\alpha} \right) < |b|$.

Relying on a Stratonovich rule for writing the perturbation does not enable us
to get a stochastic system converging to the stable equilibrium of the
deterministic ODE. Actually, we can even prove that the whole sphere but a small
northern cap is positive recurrent.

\begin{prop}
  For all $|b| > \delta > \inv{2}\left(\sqrt{4 |b|^2
  +\frac{\varepsilon^4(\alpha^2+1)^2}{\alpha^2}} -
  \frac{\varepsilon^2(\alpha^2+1)}{\alpha}   \right)$, the set $\left\{x \in
  \sphere \: : \: x \cdot b \le  \delta \right\}$ is positive
  recurrent.  Moreover, if $\tau_\delta = \inf \{ t>0 : \bmu_t \cdot b \le
  \delta \}$, 
  \begin{align*}
    \P(\tau_\delta > t \;|\; \bmu_0 \cdot b > \delta ) = O(t^{-1}).
  \end{align*}
\end{prop}

\begin{proof}
  Since $(\bmu_t)_t$ is an homogeneous Markov process, it is sufficient to prove
  that ${\E[\tau_\delta \;|\; \bmu_0 \cdot b > \delta] < +\infty}$.  Assume
  $\bmu_0 \cdot b > \delta$. Using Doob's stopping time theorem, we can write
  \begin{align*}
    \bmu_{t \wedge \tau_\delta} \cdot b - \bmu_0 \cdot b & = \int_0^{t \wedge \tau_\delta} -
    \left\{ \alpha \left( (\bmu_s \cdot b)^2 - \norm{b}^2 \right) + \varepsilon^2
    (\alpha^2 +1) \bmu_s \cdot b \right\} ds \\
    & \quad - \int_0^{t \wedge \tau_\delta} \varepsilon \big( -L(\bmu_s) b + \alpha
    ((\bmu_s \cdot b) \bmu_s -b) \big) \cdot dW_s. \\
    \E[\bmu_{t \wedge \tau_\delta} \cdot b] - \E[\bmu_0 \cdot b] & = \int_0^t \E\left[
    \ind{s \le \tau_\delta} 
    \left\{ \alpha \left( -(\bmu_s \cdot b)^2 + \norm{b}^2 \right) - \varepsilon^2
    (\alpha^2 +1) \bmu_s \cdot b \right\}\right] ds 
  \end{align*}
  For $s \le \tau_\delta$, $\bmu_s \cdot b > \delta$. Hence,
  \begin{align}
    \label{eq:Ebmu_stopping}
    \delta - \E[\bmu_0 \cdot b] \le \E[\bmu_{t \wedge \tau_\delta} \cdot b] -
    \E[\bmu_0 \cdot b] & \le  \E[t \wedge \tau_\delta] \left\{ \alpha \left(
    -\delta^2 + \norm{b}^2 \right) - \varepsilon^2 (\alpha^2 +1) \delta \right\}
  \end{align}
  For $\delta > \inv{2}\left(\sqrt{4 |b|^2
  +\frac{\varepsilon^4(\alpha^2+1)^2}{\alpha^2}} -
  \frac{\varepsilon^2(\alpha^2+1)}{\alpha}   \right)$, the term $\left\{ \alpha
  \left( -\delta^2 + \norm{b}^2 \right) - \varepsilon^2 (\alpha^2 +1)
  \delta \right\} < 0$. If we let $t$ go to infinity in
  Equation~\eqref{eq:Ebmu_stopping}, we deduce that $\lim_{t \rightarrow +\infty}
  \E[t \wedge \tau_\delta] < +\infty$. Using the monotone convergence theorem, we obtain
  that $ \E[\lim_{t \rightarrow +\infty}t \wedge \tau_\delta] = \lim_{t \rightarrow
  +\infty} \E[t \wedge \tau_\delta]   < +\infty$. Hence,
  \begin{align}
    \label{eq:lim_strato}
    \lim_{t \rightarrow +\infty} \E[\tau_\delta \ind{\tau_\delta < t}] + t \P(\tau_\delta > t) <
    +\infty 
  \end{align}
  Moreover $\lim_{t \rightarrow +\infty} \E[\tau_\delta \ind{\tau_\delta < t}] =  \E[\tau_\delta
  \ind{\tau_\delta < +\infty}]$ and $\lim_{t \rightarrow +\infty} \P(\tau_\delta > t) = \P(\tau_\delta
  = +\infty)$. This implies that $\P(\tau_\delta < +\infty) = 1$. If we plug this result
  back into Equation~\eqref{eq:lim_strato}, we also deduce that ${\E[\tau_\delta] < +\infty}$.
\end{proof}

\paragraph{The physical phenomena modelled by the Itô and the Stratonovich
rules.} Whereas, these tow kinds of stochastic perturbations both model thermal
effects, these are of completely different natures. As the time homogeneous
dynamics obtained from the Stratonovich rule suggests it, the physical
environment does not evolve with time, meaning that the temperature remains
constant over time. Hence, the Stratonovich approach is well suited to model a
constant injection of heat into the system. 

On the contrary, the Itô approach describes the evolution of a model starting
with a fixed amount of heat with no more heat injection after time $0$. This
slow decrease of the system temperature over time explains why we came up with
an SDE with time decreasing coefficients.

\section{Main results: long time behaviour}
\label{sec:long_time}

\subsection{Almost sure convergence}

In this part, we prove the almost sure convergence of $\mu_t$ to $b / |b|$ when
$t$ goes to infinity. This is achieved by studying the pathwise behaviour of the
process $(\mu_t \cdot b)_t$.

\begin{thm}
  \label{thm:cv-ps}
  $\displaystyle \lim_{t \longrightarrow \infty} \mu_t \cdot b = \norm{b}$ a.s.
\end{thm}

To prove this result, we need a preliminary result stating that the stochastic
integral in Equation~\eqref{eq:sde_scalar} actually vanishes at
infinity.
\begin{lemma}
  \label{lem:martingale}
  $$
  \sup_t \int_0^t \frac{1}{h(u)} \big( -\mu_u \wedge  b + \alpha ((\mu_u \cdot
  b) \mu_u -b) \big) \cdot dW_u < \infty \quad a.s.
  $$
\end{lemma}
\vskip1em

\begin{proof}[Proof of Theorem~\ref{thm:cv-ps}]
  From Lemma~\ref{lem:martingale}, we know that 
  \begin{align*}
    \sup_t \int_0^t \frac{1}{h(u)} \big( -\mu_u \wedge  b + \alpha ((\mu_u \cdot
    b) \mu_u -b) \big) \cdot dW_u < \infty \quad a.s.
  \end{align*}
  Hence, we can define for all $t \ge 0$
  \begin{align*}
    X_t = \mu_t \cdot b - \int_t^\infty \frac{\varepsilon}{h(u)} \big( -\mu_u
    \wedge  b + \alpha ((\mu_u \cdot b) \mu_u -b) \big) \cdot dW_u 
  \end{align*}
  Let $\norm{b} > \delta >0$ be chosen close to $0$. There exists $T$ such that
  for all $t \ge T$, $\norm{X_t - \mu_t \cdot b} \le \delta$. Moreover from
  Equation~\eqref{eq:sde_scalar}, we can deduce that for all $t > s > T$
  \begin{align}
    \label{eq:X_eds}
    X_t - X_s =  \int_s^t - \mu_u \cdot b \frac{h'(u)}{h(u)}
    -  \frac{\alpha}{h(u)} ( (\mu_u \cdot b)^2 - \norm{b}^2 ) du.
  \end{align}
  Let $\eta < \norm{b}$ be chosen close to $\norm{b}$.  On the set
  $\{0<\mu_u\cdot b <\eta\}$ we have $\norm{\mu_u\wedge b}^2\geq
  \norm{b}^2-\eta^2$. Thus,
  $$
  -  (\mu_u \cdot b) \frac{h'(u)}{h(u)} + \frac{\alpha}{h(u)} \norm{\mu_u \wedge
  b}^2 \ge - \norm{b}  \frac{h'(u)}{h(u)} + \alpha \frac{\norm{b}^2-\eta^2}{h(u)}.
  $$
  We can always choose $T$ such that for all $u > T$,
  $$
  - \norm{b}  \frac{h'(u)}{h(u)} + \alpha \frac{\norm{b}^2-\eta^2}{h(u)}\geq
  \alpha \frac{\norm{b}^2-\eta^2}{2h(u)}.
  $$
  Hence, on the set $\{0<\mu_u\cdot b <\eta\}$,
  \begin{align}
    \label{eq:lower1}
    -  (\mu_u \cdot b) \frac{h'(u)}{h(u)} + \frac{\alpha}{h(u)} \norm{\mu_u \wedge
  b}^2 \ge \alpha \frac{\norm{b}^2-\eta^2}{2h(u)}.
  \end{align}
  On the set $\{\mu_u \cdot b \le 0\}$ we have 
  \begin{align}
    \label{eq:lower2}
    - (\mu_u \cdot b) \frac{h'(u)}{h(u)} + \frac{\alpha}{h(u)} \norm{\mu_u \wedge
    b}^2 & \ge ( - (\mu_u \cdot b) + \alpha \norm{\mu_u \wedge
    b}^2) \frac{h'(u)}{h(u)} \nonumber \\
    &  \ge  \min(\norm{b}, \alpha \norm{b}^2) \frac{1}{2}\frac{h'(u)}{h(u)}.
  \end{align}

  The last inequality comes from the fact that if $\pi/2 \leq x \leq 3\pi/2$, we have
  either $-\cos(x) \ge \sqrt{2}/2$ or  $|\sin(x)| \ge \sqrt{2}/2$.

  Therefore, by combining Equations~\eqref{eq:lower1} and~\eqref{eq:lower2}, we
  find that there exists $\bar T \ge T$, such that for all $t \ge \bar T$,  on
  the event $\{\mu_t \cdot b < \eta\}$ we have   $- (\mu_t \cdot b)
  \frac{h'(t)}{h(t)} + \frac{\alpha}{h(t)} \norm{\mu_t \wedge b}^2 \ge c
  \frac{h'(t)}{h(t)}$ where $c$ is a positive real constant depending on $\eta$.
  Therefore, we deduce from Equation~\eqref{eq:X_eds}
  \begin{align*}
    X_t - X_s & \ge - \norm{b}\int_s^t \frac{h'(u)}{h(u)} \ind{\mu_u \cdot b > \eta}
    du + c \int_s^t  \frac{h'(u)}{h(u)} \ind{\mu_u \cdot b \le \eta} du \\
    & \ge - \norm{b}\int_s^t \frac{h'(u)}{h(u)} \ind{X_u > \eta - \delta}
    du + c \int_s^t  \frac{h'(u)}{h(u)} \ind{X_u \le \eta - \delta} du.
  \end{align*}
  The values of $\delta$ and $\eta$ can always be chosen that $0 < \eta - 2
  \delta < \norm{b}$.

  Assume that $X_s \le  \eta - 2 \delta$. Then, for all $t \ge s$ such that for
  all $u \in [s,t]$,  $X_u \le \eta - \delta$, we have
  \begin{align*}
    X_t - X_s & \ge c \int_s^t  \frac{h'(u)}{h(u)} du
  \end{align*}
  As $h'/h$ is not integrable, $t$ must be finite, which means there exists $t
  \ge s$ such that $X_t > \eta - \delta > \eta - 2 \delta$.  Hence, we can
  assume that $X_s >  \eta - 2 \delta$. 
  
  Either, for all $t \ge s $, $X_t > \eta - 2 \delta$, or thanks to the
  continuity of $X$ there exists $t_0$ for which $X_{t_0} = \eta - 2 \delta$ and
  we have just seen that, in this case, there exists $t \ge t_0$ such that $X_t
  \ge \eta - \delta$ and for all $u \in [t_0, t]$ $\eta -  \delta \ge X_u > \eta
  - 2 \delta$. This reasoning enables us to prove that for all $t \ge s$, $X_t
  \ge \eta - 2 \delta$, ie.  $\mu_t \cdot b \ge \eta - 3 \delta$.  As $\eta$ can
  be chosen arbitrarily close to $\norm{b}$ and $\delta$ arbitrarily small, this
  proves the almost sure convergence of $\mu_t \cdot b$ to $\norm{b}$.
\end{proof}

\begin{proof}[Proof of Lemma~\ref{lem:martingale}]
  From Doob's inequality we have
  \begin{align}
    \label{eq:doob}
    \E\left[ \sup_t \norm{\int_0^t \frac{1}{h(u)} \big( -\mu_u \wedge  b + \alpha ((\mu_u \cdot
    b) \mu_u -b) \big) \cdot dW_u }^2 \right] \le \nonumber \\
    \int_0^\infty \frac{1}{h(u)^2} \E\left[ (\norm{b}^2 - \norm{\mu_u \cdot
    b}^2)  \right] (1+\alpha^2) du.
  \end{align}
  Now, we will prove that the r.h.s is finite.

  From Equation~\eqref{eq:sde_scalar}, we get after integrating and taking the
  expectation for all $t > 0$
  \begin{align*}
    \E[\mu_t \cdot b] - \E[\mu_0 \cdot b] & = - \int_0^t \E[\mu_u \cdot b] \frac{h'(u)}{h(u)} 
    -  \frac{\alpha}{h(u)} \E\left[ (\mu_u \cdot b)^2 - \norm{b}^2 \right] du  \\
    \E[\mu_t \cdot b]' h(t) + \E[\mu_t \cdot b] h'(t) & =  \alpha
    \E\left[ \norm{b}^2 - (\mu_t \cdot b)^2 \right] \\
    \E[\mu_t \cdot b] - \frac{h(s)}{h(t)} \E[\mu_0 \cdot b] & =
    \frac{\alpha}{h(t)} \int_0^t  \E\left[ \norm{b}^2 - (\mu_u \cdot b)^2
    \right] du
  \end{align*}
  Hence, we deduce that
  \begin{align*}
    \sup_t \frac{1}{h(t)} \int_0^t  \E\left[ \norm{b}^2 - (\mu_u \cdot b)^2
    \right] du < \frac{2 \norm{b}}{\alpha} = \kappa.
  \end{align*}
  Let us consider the upper--bound in Equation~\eqref{eq:doob} truncated to $t$
  and perform an integration by parts to obtain
  \begin{align*}
    \int_0^t &\frac{1}{h(u)^2} \E\left[ (\norm{b}^2 - \norm{\mu_u \cdot
    b}^2)  \right]  du  \\
    = & \left[ \frac{1}{h(u)^2} \int_0^u \E\left[ (\norm{b}^2 -
    \norm{\mu_v \cdot b}^2)  \right] dv \right]_0^t + \int_0^t \frac{2
    h'(u)}{h(u)^3} \int_0^u \E\left[ (\norm{b}^2 -
    \norm{\mu_v \cdot b}^2)  \right] dv du   \\
    & \le \kappa \frac{1}{h(t)}  + 2 \kappa \int_0^t \frac{h'(u)}{h(u)^2} du \\
    & \le \kappa \left( \inv{h(t)} + 2 \left( \inv{h(0)} - \inv{h(t)} \right)
    \right) \le \frac{2 \kappa}{h(0)}
  \end{align*}
  This proves that the r.h.s of Equation~\eqref{eq:doob} is finite and ends the
  proof of Lemma~\ref{lem:martingale}.
\end{proof}

\subsection{Convergence rate}
\label{ss-cvgce-rate}

In this section, we are interested in the behaviour of $h(t) (\norm{b} - \mu_t
\cdot b)$.  In particular, we establish the rate of decrease of the $\L^1$ norm
of $\norm{b} - \mu_t \cdot b$ to zero.

\begin{thm}
  \label{thm:E_mu_rate}
  $\displaystyle \lim_{t \longrightarrow \infty} \E[ h(t) \norm{\norm{b} - \mu_t \cdot b} ] =
  \frac{ \varepsilon^2 (1+\alpha^2)}{2\alpha}$.
\end{thm}

To prove this Theorem, we need the following lemma.
\begin{lemma}
  \label{lem:speed}
  $\displaystyle \lim_{t \rightarrow +\infty} \E \left[ h(t) (|b| - \mu_t \cdot
  b)^2 \right] = 0$.
\end{lemma}

\begin{proof}[Proof of Lemma~\ref{lem:speed}]
  We define the process $X$ by $X_t = h(t)   (|b| - \mu_t \cdot
  b)^2$ and apply Itô's formula to find
  \begin{align*}
    dX_t & = h'(t)  (|b| - \mu_t \cdot b)^2 dt + 2 h(t)  (|b| - \mu_t \cdot b)
    d(-\mu_t \cdot b) + h(t) d<\mu \cdot b>_t \\
    &= \frac{h'(t)}{h(t)} X_t dt +
    2 (|b| - \mu_t \cdot b) \left( h'(t) (\mu_t \cdot b) + \alpha (|b|^2 -
    (\mu_t \cdot b)^2) \right) dt  \\
    & + 2 \varepsilon (|b| - \mu_t \cdot b) \left(
    - \mu_t \wedge b + \alpha ((\mu_t \cdot b)\mu_t - b) \right) dW_t +
    \frac{\varepsilon^2 (\alpha^2 + 1)}{h(t)} (|b|^2 - (\mu_t \cdot b)^2) dt.
  \end{align*}
  Then, we integrate and take expectation to obtain
  \begin{align*}
    \E[X_t]' & = \frac{h'(t)}{h(t)} \E[X_t]  + \E \left[ 2 h'(t) (|b| - \mu_t \cdot
    b) (\mu_t \cdot b) \right] + 2 \alpha \E \left[(|b| - \mu_t \cdot b) (|b|^2 -
    (\mu_t \cdot b)^2) \right]\\
    & \quad + h'(t) \E\left[ |b|^2 - (\mu_t \cdot b)^2 \right] \\
    &= - \frac{h'(t)}{h(t)} \E[X_t] + 2 h'(t) |b| \E[(|b| - \mu_t \cdot b)] - 2
    \frac{\alpha}{h(t)} \E \left[(|b| + \mu_t \cdot b) X_t \right]
    + h'(t) \E\left[ |b|^2 - (\mu_t \cdot b)^2 \right] \\
    & = -2 \frac{\alpha}{\varepsilon^2 (\alpha^2+1)} h'(t) |b| \E[X_t] - 
    2 \frac{\alpha}{\varepsilon^2 (\alpha^2+1)} h'(t) \E[X_t (\mu_t \cdot b)]
    - \frac{h'(t)}{h(t)} \E[X_t] \\
    & \quad + h'(t) \E[(|b| - \mu_t \cdot b) (3 |b| + \mu_t
    \cdot b)] 
  \end{align*}
  Let $\beta = 2 |b| \frac{\alpha}{\varepsilon^2 (\alpha^2+1)}$. We integrate the
  previous equation to find
  \begin{align*}
    \E[X_t] - \E[X_0] \expp{-\beta (h(t) - h(0))} & = \expp{-\beta h(t)} \int_0^t
    \Bigg(-\frac{\beta}{|b|} h'(s) \E[X_s (\mu_s \cdot b)] - h'(s) \E[(|b| - \mu_s
    \cdot b)^2] \\
    & \qquad + h'(s) \E[(|b| - \mu_s \cdot b) (3 |b| + \mu_s
    \cdot b)] \Bigg) \expp{\beta h(s)} ds \\
    & = \expp{-\beta h(t)} \int_0^t
    \Bigg(-\frac{\beta}{|b|} h'(s) \E[X_s (\mu_s \cdot b)] - h'(s) \E[(|b| - \mu_s
    \cdot b)^2] \\
    & \qquad + h'(s) \E[(|b| - \mu_s \cdot b) (3 |b| + \mu_s
    \cdot b)] \Bigg) \expp{\beta h(s)} ds.
  \end{align*}
  From Theorem~\ref{thm:cv-ps} combined with the bounded convergence theorem, we
  know that $\E[(|b| - \mu_t \cdot b)^2] $ and $\E[(|b| - \mu_t \cdot b) (3 |b|
  + \mu_t \cdot b)]$ tend to zero when $t$ goes to infinity. Then, we can apply
  Lemma~\ref{lem:intexp} to show that 
  \begin{align*}
   \expp{-\beta h(t)} \int_0^t
    h'(s) \Bigg( - \E[(|b| - \mu_s
    \cdot b)^2] + \E[(|b| - \mu_s \cdot b) (3 |b| + \mu_s
    \cdot b)]  \Bigg) \expp{\beta h(s)} ds \xrightarrow[t \rightarrow
    +\infty]{} 0.
  \end{align*}
  Hence, we get
  \begin{align*}
    \limsup_{t \rightarrow +\infty} \E[X_t] &=-\frac{\beta}{|b|} \liminf_{t
    \rightarrow +\infty} \expp{-\beta h(t)} \int_{h(0)}^{h(t)} \E[X_{h^{-1}(v)}
    (\mu_{h^{-1}(v)} \cdot b)] \expp{\beta v} dv \\
    & \le -\frac{\beta}{|b|} \liminf_{t \rightarrow +\infty} \E[X_{t}
    (\mu_{t} \cdot b)]
  \end{align*}
  where we have used Lemma~\ref{lem:liminfexp} for the last inequality. Finally,
  Fatou's Lemma yields
  \begin{align*}
    \limsup_{t \rightarrow +\infty} \E[X_t]     & \le -\frac{\beta}{|b|}
    \E[\liminf_{t \rightarrow +\infty} X_{t} (\mu_{t} \cdot b)] 
    \le -\beta \E[\liminf_{t \rightarrow +\infty} X_{t} ] \le 0.
  \end{align*}
  Since, $X_t \ge 0$ a.s., we conclude that $\lim_{t \rightarrow +\infty}
  \E[X_t] =0$.
\end{proof}

\begin{proof}[Proof of Theorem~\ref{thm:E_mu_rate}]
  As $\norm{b} - \mu_t \cdot b \ge 0$, the $\L^1$ norm boils down to a basic
  expectation.

  Let us define $\xi_t = \norm{b} - \mu_t \cdot b$ for $t \ge 0$. From
  Equation~\eqref{eq:sde_scalar}, we get
  \begin{align*}
    d\xi_t & = - \xi_t \frac{h'(t)}{h(t)} dt + \norm{b} \frac{h'(t)}{h(t)} dt +
    \frac{\alpha}{h(t)}  \left( (\mu_t \cdot b)^2 - \norm{b}^2 \right) dt -
    \frac{\varepsilon}{h(t)} \big( - \mu_t \wedge b + \alpha ((\mu_t \cdot b)
    \mu_t -b)  \big) \cdot dW_t \\
    & = - \xi_t \frac{h'(t)}{h(t)} dt + \norm{b} \frac{h'(t)}{h(t)} dt -
    \frac{\alpha}{h(t)}  \xi_t (2 \norm{b} - \xi_t) dt -
    \frac{\varepsilon}{h(t)} \big( - \mu_t \wedge b + \alpha ((\mu_t \cdot b)
    \mu_t -b)  \big) \cdot dW_t 
  \end{align*}
  If we introduce $Z_t = h(t) ( \norm{b} - \mu_t \cdot b)$, we can write
  \begin{align*}
    dZ_t = \left( h'(t) \norm{b} - \alpha \xi_t (2 \norm{b} - \xi_t)  \right) dt
    -
    \varepsilon \big( - \mu_t \wedge b + \alpha ((\mu_t \cdot b) \mu_t -b)  \big)
    \cdot dW_t 
  \end{align*}
  From the dynamics of $Y$, we deduce that $\E[Z_t]$ solves the following
  differential equation
  \begin{align*}
    \E[Z_t]' & = \norm{b} h'(t) - \alpha (2 \norm{b} \E[\xi_t] - \E[\xi_t^2]) \\
    \E[Z_t]' & = \norm{b} h'(t) - \alpha 2 \norm{b} \frac{\E[Z_t]}{h(t)} +
    \alpha  \E[\xi_t^2] \\
    \E[Z_t]' & = \norm{b} h'(t) - \alpha \frac{2 \norm{b}}{\varepsilon^2
    (\alpha^2+1) } h'(t) \E[Z_t]  + \alpha \E[\xi_t^2] \\
    \left (\E[Z_t] \expp{ \int_0^t \frac{2\alpha \norm{b}}{\varepsilon^2
    (\alpha^2+1) } h'(u) du} \right)' & = \left( \norm{b} h'(t) + \alpha
    \E[\xi_t^2]  \right) \expp{ \int_0^t \frac{2 \alpha \norm{b}}{\varepsilon^2
    (\alpha^2+1) } h'(u) du}
  \end{align*}
  Now, we can integrate the previous equation to obtain
  \begin{align*}
    \E[Z_t] \expp{ \frac{2\alpha \norm{b}}{\varepsilon^2 (\alpha^2+1) } (h(t) - h(0))}
    - \E[Z_0] & = \norm{b}  \int_0^t h'(u) \expp{ \frac{2 \alpha
    \norm{b}}{\varepsilon^2 (\alpha^2+1) } (h(u) - h(0))} du \\
    & \quad + \alpha \int_0^t
    \E[\xi_u^2]  \expp{ \frac{2\alpha \norm{b}}{\varepsilon^2 (\alpha^2+1) } (h(u) -
    h(0))} du \\
    \E[Z_t]     - \E[Z_0] \expp{ - \frac{2 \alpha\norm{b}}{\varepsilon^2 (\alpha^2+1)
    } (h(t) - h(0))} & = \frac{\varepsilon^2(\alpha^2+1)}{2 \alpha} \left( 1 -
    \expp{-\frac{2\alpha \norm{b}}{\varepsilon^2 (\alpha^2+1) } (h(t) - h(0))} \right)
    \\ & \quad + \alpha \expp{ - \frac{2\alpha \norm{b}}{\varepsilon^2 (\alpha^2+1) }
    (h(t) - h(0))} \int_0^t \E[\xi_u^2]  \expp{ \frac{2\alpha \norm{b}}{\varepsilon^2
    (\alpha^2+1) } (h(u) - h(0))} du
  \end{align*}
  From Theorem~\ref{thm:cv-ps},  we know that $\xi_t$ tends to $0$ a.s., therefore the bounded
  convergence theorem yields that $\lim_{u \longrightarrow \infty}
  \E[\xi_u^2] = 0$. Hence, as $h(t)$ tends to infinity with $t$, it is easy to show that 
  \begin{align*}
    \lim_{t \longrightarrow \infty}  \expp{ - \frac{2 \norm{b}}{\varepsilon^2
    (\alpha^2+1) } (h(t) - h(0))} \int_0^t \E[\xi_u^2]  \expp{ \frac{2
    \norm{b}}{\varepsilon^2 (\alpha^2+1) } (h(u) - h(0))} du = 0.
  \end{align*}
  Then, we can deduce that
  \begin{align*}
    \lim_{t \longrightarrow \infty} \E[Z_t] = \frac{\varepsilon^2(\alpha^2+1)}{2
    \alpha}. 
  \end{align*}
\end{proof}

As a corollary of Theorem~\ref{thm:E_mu_rate}, we can prove the following
results using Markov's inequality.
\begin{corollary}
  For all $0 < \beta < 1/2$ and $\eta > 0$,
  $\P( t^\beta (|b| - \mu_t \cdot b) \ge \eta) \longrightarrow 0$.
\end{corollary}

\section{Hysteresis phenomena}
\label{sec:hyster}

In this section, we want to study the impact of the stochastic perturbation on
the reversibility of the system; we are wondering whether the stochastic part
may induce an hysteresis phenomenon. In order to observe this, the particle is
submitted to an external field linearly varying from $+\bb$ to $-\bb$ where $\bb
\in \sphere$ and with constant direction and bounded modulus. We have seen in
Section~\ref{sec:long_time} that when the external field is fixed, the magnetic
moment $\mu$ asymptotically stabilizes along this field. If the
external field varies sufficiently slowly compared to the stabilization rate of
$\mu$, we expect that $\mu$ will take different back an forth paths when the
external field switches from $+\bb$ to $-\bb$ and then from $-\bb$ to $+\bb$:
this characterizes the hysteretic behaviour of the system.

In order to highlight this property, we will study the evolution of a suitably
rescaled system on the time interval $[0,1]$ and show that the average back and
forth paths of $\mu_t \cdot b$  can not cross at the point $t=1/2$. \\

We consider a two time scale model: a slower scale for the variations of the
external field and a faster scale for the Landau Lifshitz evolution of the
magnetic moment.

Let $\eta >0$ be a fixed time scale and $\bb \in \sphere$ the direction of the
external field. We define the external filed $b^\eta$ linearly varying between
$+\bb$ and $-\bb$ on the interval $[0,1/\eta]$ by
\begin{equation*}
  b^\eta(t)  = (1 - 2 t \; \eta) \; \bb \quad \text{for } t \in [0,1/\eta]
\end{equation*}
We assume that the magnetic moment $\mu^\eta$ is affected by $b^\eta(t)$
according to the following equation for $t \in [0,1/\eta]$
\begin{equation*}
  \begin{cases}
    dY^\eta_t &=  - \mu^\eta_t \wedge (b^\eta(t)\;  dt + \varepsilon\; dW_t) -
    \alpha \mu^\eta_t \wedge \mu^\eta_t \wedge (b^\eta(t) \; dt + \varepsilon \; dW_t)  \\
    \mu^\eta_t & = \frac{Y^\eta_t} {|Y^\eta_t|}\\
    Y^\eta_0&= \bb
  \end{cases}
\end{equation*}
In order to work on the interval $[0,1]$, we introduce rescaled versions of both
the external field and the magnetic moment defined for $t \in [0,1]$.
\begin{align*}
  b(t) = b^\eta(t/\eta), \qquad  Z^\eta_t = Y^\eta_{t/ \eta}, \qquad
  \lambda^\eta_t = \mu^\eta_{t/\eta}.
\end{align*}
Using the time scale property of the stochastic integral, we can write
\begin{align*}
  dZ^\eta_t &=  - \lambda^\eta_t \wedge \left(b(t) \; \frac{1}{\eta} dt + \varepsilon\;
  dW_{t/\eta}\right) - \alpha  \lambda^\eta_t \wedge \lambda^\eta_t \wedge \left(b(t) \;
  \frac{1}{\eta}  dt + \varepsilon \; dW_{t / \eta}\right)  
\end{align*}
From the scaling property of the Brownian motion, we know that $(\sqrt{\eta} W_{t
/ \eta})$ is still a Brownian motion. So we get
\begin{align*}
  dZ^\eta_t &=  - \lambda^\eta_t \wedge \left(b(t) \; \frac{1}{\eta} dt +
  \varepsilon\; \frac{1}{\sqrt{\eta}} dW_t\right) - \alpha \lambda^\eta_t \wedge
  \lambda^\eta_t \wedge \left(b(t) \; \frac{1}{\eta}  dt + \varepsilon
  \frac{1}{\sqrt{\eta}} \; dW_t\right)  
\end{align*}
It is important to notice that the factor $\eta$ acts as a time scale parameter
for the deterministic part, but that the corresponding scaling parameter for the
stochastic part is $\sqrt{\eta}$.

Following the proof of Proposition~\ref{prop:norm_Y}, it is obvious to show that
$d\norm{Z^\eta_t}^2 = 2 (1+\alpha^2) \varepsilon^2 / \eta \; dt$. Then, we introduce
for $t \in [0,1]$
\begin{align*}
  h^\eta(t) = |Z^\eta_t| = \sqrt{2 (1+\alpha^2) \varepsilon^2 t / \eta + 1}.
\end{align*}

The Ito formula applied to $(\lambda^\eta_t \cdot \bb)_t$ yields
\begin{align*}
  d( \lambda^\eta_t \cdot \bb) & =  - (\lambda^\eta_t \cdot \bb)
  \frac{{h^\eta}^{\prime} (t)}{h^\eta(t)} dt + \alpha
  \frac{1-2t}{\eta h^\eta(t)} \norm{\lambda^\eta_t \wedge \bb}^2 dt  \nonumber \\ 
  & \qquad - \frac{\varepsilon}{\sqrt{\eta} h^\eta(t)} \left( (\lambda^\eta_t
  \wedge  \; dW_t) \cdot \bb + \alpha (\lambda^\eta_t \cdot \bb) 
  (\lambda^\eta_t \cdot dW_t) - \alpha (\bb \cdot dW_t) \right) 
\end{align*}
The stochastic part vanishes when taking expectation as in the previous
section to find
\begin{align}
  \label{eq:E_lambdat_b}
  \E( \lambda^\eta_t \cdot \bb) - \lambda_0 \cdot \bb & =  \int_0^t -
  \E(\lambda^\eta_u \cdot \bb) 
  \frac{{h^\eta}^{\prime} (u)}{h^\eta(u)} +
  \alpha \frac{1-2u}{\eta h^\eta(u)} \E\norm{\lambda^\eta_u \wedge \bb}^2 du  
\end{align}

\begin{prop}
  \label{prop:hyster} 
  For all $t \in [0, 1/2]$, 
  \begin{align*}
    \E(\lambda^\eta_t \cdot \bb) \ge \frac{1}{h^\eta(t)} \ge \frac{1}{\sqrt{1 +
    \frac{(1+\alpha^2) \varepsilon^2}{\eta}}}.
  \end{align*}
\end{prop}

\begin{proof}
  Since the process $\lambda^\eta$ is pathwise continuous and bounded, it is easy to
  show that the deterministic function $e(t): t \longmapsto \E(\lambda^\eta_t \cdot
  b)$ is of class $C^1$. Hence, we can differentiate
  Equation~\eqref{eq:E_lambdat_b}
  \begin{align*}
    e'(t) =  - e(t) \frac{{h^\eta}^{\prime} (t)}{h^\eta(t)} +
    \alpha \frac{1-2t}{\eta h^\eta(t)} \E\norm{\lambda^\eta_t \wedge \bb}^2 
  \end{align*}
  As $t \le 1/2$, the second term on the r.h.s is non--negative, hence
  \begin{align*}
    e'(t) \ge   - e(t) \frac{{h^\eta}^{\prime} (t)}{h^\eta(t)} 
  \end{align*}
  From this inequality, we deduce that $(e(t) h^\eta(t))' \ge 0$, which leads to
  the following lower bound
  \begin{align*}
    e(t) \ge \frac{1}{h^\eta(t)} \quad \text{for } t \le 1/2
  \end{align*}
\end{proof}

\section{Numerical experiments}

In this section, we want to illustrate the theoretical results obtained in
Sections~\ref{sec:long_time} and~\ref{sec:hyster} using some numerical
simulations.

\paragraph{Long time behaviour.} We consider the stochastic
system~\eqref{eq:sys_sto} and discretize it with the help of an Euler scheme
$(\bar{Y},\bar{\mu})$, on a time grid with step size $\delta t>0$.

Figure~\ref{fig:cv_ps_eds} shows the long time behaviour of $(\bar{\mu}_t\cdot
b)_{t\geq 0}$,  for one path of the scheme $(\bar{Y},\bar{\mu})$, with time
step size $\delta t = 0.01$ and for different values of the damping parameter
$\alpha$. The parameter $\varepsilon$ is fixed to $0.1$, we have taken
$\norm{b}=1$, and set $\mu_0=-b$.  The almost sure convergence of $\mu_t \cdot b$
to $|b|$, as stated by Theorem~\ref{thm:cv-ps}, is well illustrated by
Figure~\ref{fig:cv_ps_eds} and one can also see how the parameter $\alpha$
impacts the characteristic time of the system, ie. the time needed to stabilize
around the limit.

Now, we wish to compare the rates of convergence studied in  Subsection
\ref{ss-cvgce-rate} to numerical observations. From Theorem \ref{thm:E_mu_rate},
$\frac{2\alpha \sqrt{2}}{\varepsilon \sqrt{(1+\alpha^2)}} \sqrt{t} \;
\E(\norm{b} - \mu_t.b)$ converges to $1$ when $t$ goes to infinity. This is
illustrated by Figure \ref{fig:cv_rate_E_eds} for different values of $\alpha$.
This figure confirms that decreasing the parameter $\alpha$ leads to a decrease
of the convergence rate of $\E(\norm{b} - \mu_t.b)$.

\begin{figure}[!ht]
  \begin{center}
    \includegraphics[scale=0.6]{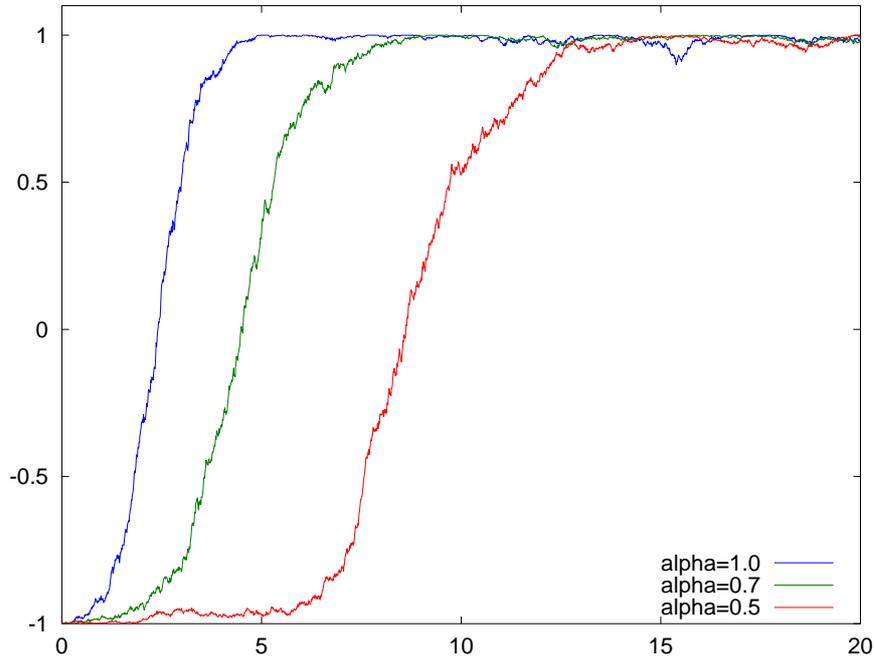}
  \end{center}
  \caption{Almost sure convergence of $\mu_t \cdot b$ with $\mu_0 = -b$,
  $\norm{b}=1$, $\varepsilon = 0.1$.}
  % 2,000 pas de temps pour T=20
  \label{fig:cv_ps_eds}
\end{figure}

\begin{figure}[!ht]
  \begin{center}
    \includegraphics[scale=0.6]{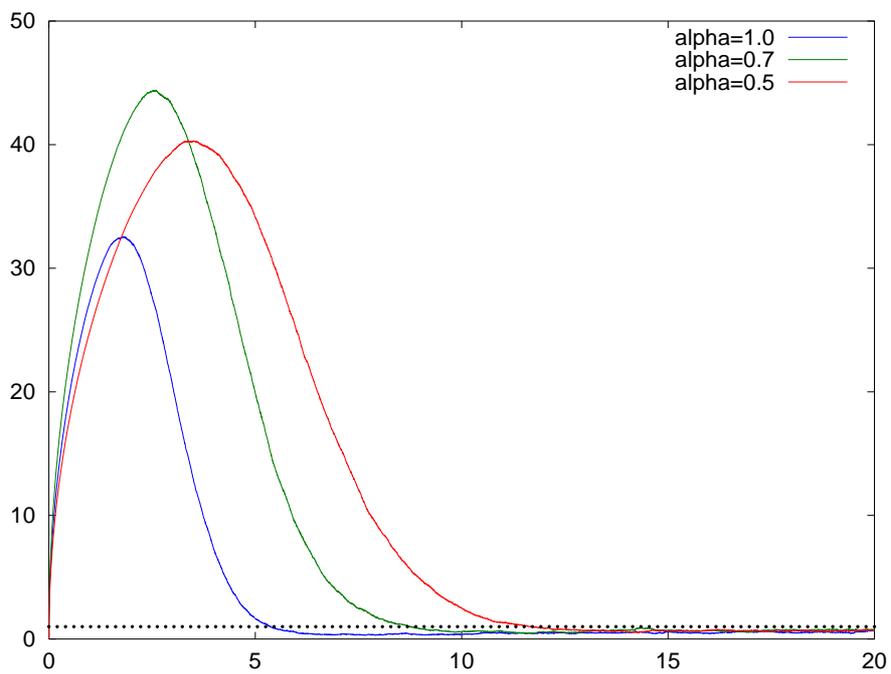}
  \end{center}
  \caption{Convergence of $\frac{2\alpha \sqrt{2}}{\varepsilon
  \sqrt{(1+\alpha^2)}} \sqrt{t} \; \E(\norm{b} - \mu_t.b)$ with $\mu_0 =
  -b$, $\norm{b}=1$ and $\varepsilon = 0.1$. The horizontal dashed
  line is at level one. The expectation is computed using a Monte--Carlo
  method with $100$ samples.}
  % 2,000 pas de temps pour T=20
  \label{fig:cv_rate_E_eds}
\end{figure}

\paragraph{Hysteresis phenomena.}

\begin{figure}[!htbp]
  \begin{center}
    \includegraphics[scale=0.6]{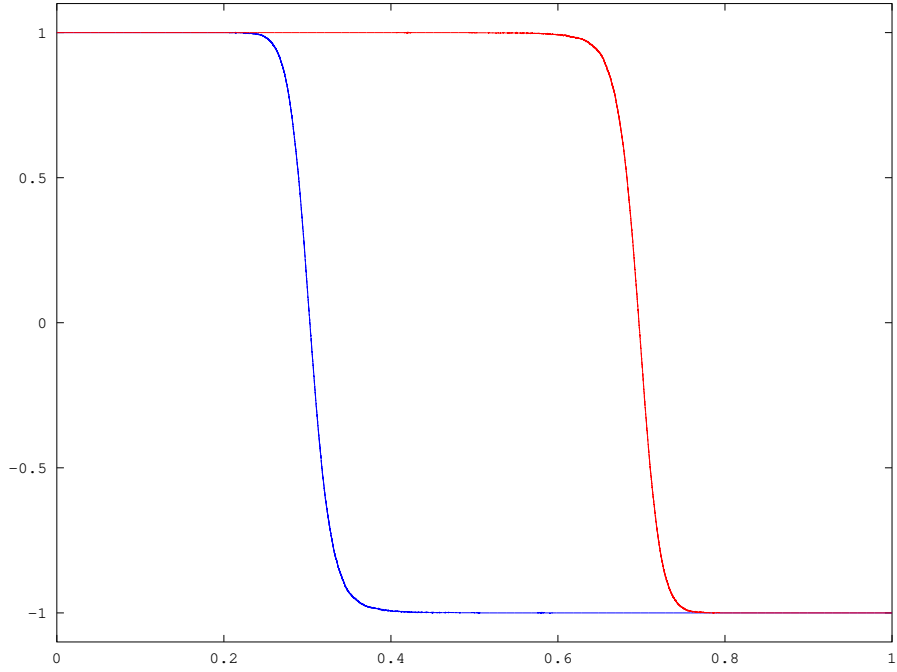}
  \end{center}
  \caption{Pathwise hysteresis phenomena with $\alpha=1$, $\varepsilon = 0.005$
  and $\eta = 0.01$. The red curve is the forward path whereas the blue curve is
  the backward path. The $x$-axis is time.}
  % 500,000 pas de temps 
  \label{fig:hyster_1}
\end{figure}

\begin{figure}[!htbp]
  \begin{center}
    \includegraphics[scale=0.6]{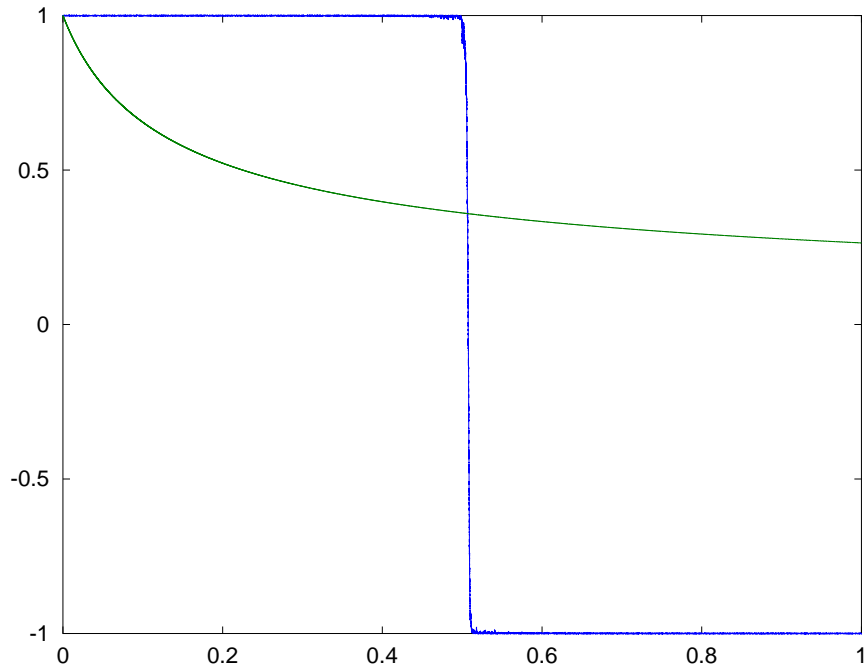}
  \end{center}
  \caption{Pathwise hysteresis phenomena with 
  $\alpha=1$, $\varepsilon = 0.01$ and $\eta = 3.1E-5$. The blue curve is the
  evolution of $\mu_t \cdot \bb$ and the green curve is $1/h^\eta(t)$. The
  $x$-axis is time.}
  % 500,000 pas de temps 
  \label{fig:hyster_2}
\end{figure}
\begin{figure}[!htbp]
  \begin{center}
    \includegraphics[scale=0.6]{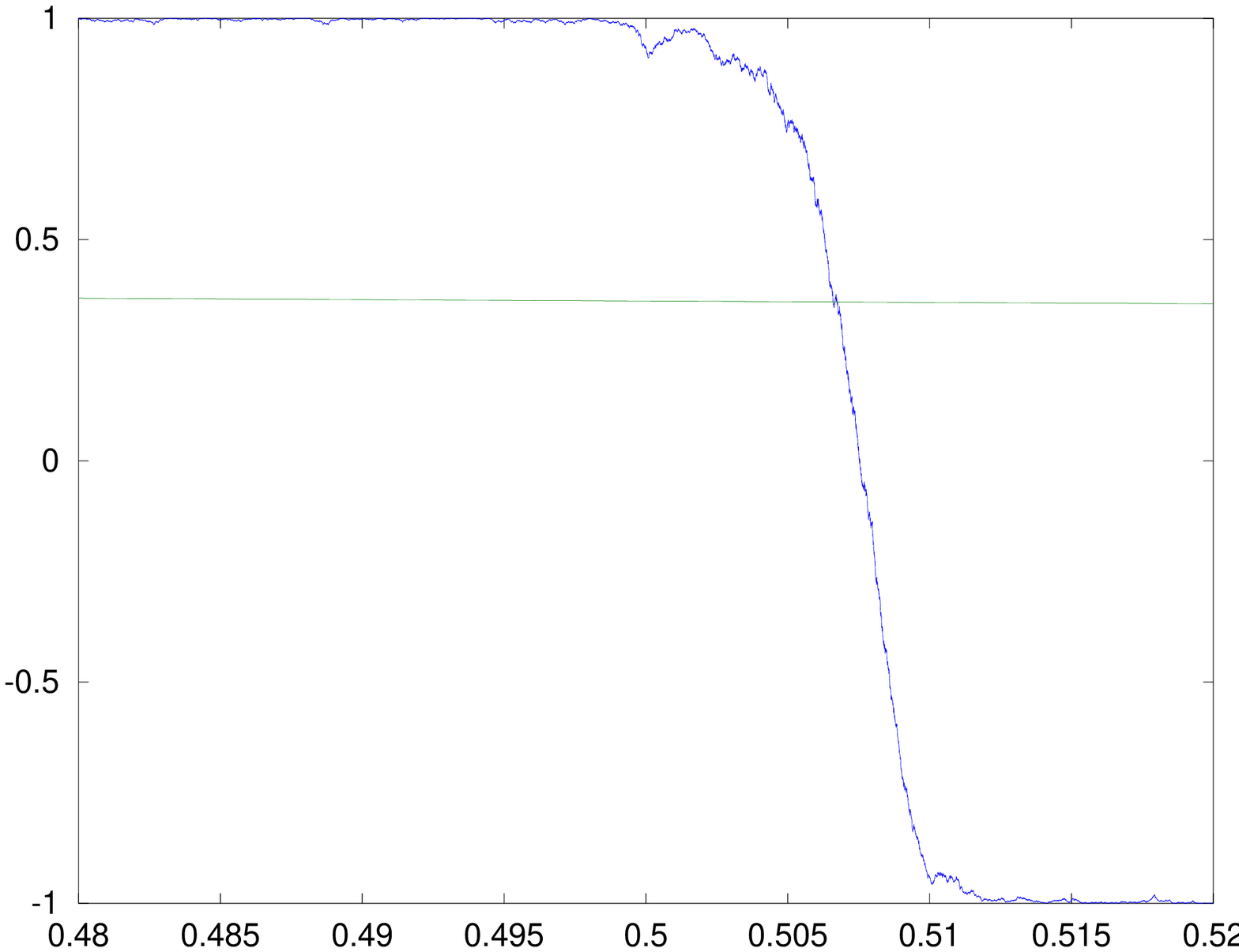}
  \end{center}
  \caption{Zoom of Figure~\ref{fig:hyster_2} around $t=1/2$. The blue curve is the
  evolution of $\mu_t \cdot \bb$ and the green curve is $1/h^\eta(t)$.}
  % 500,000 pas de temps 
  \label{fig:zoom_hyster_2}
\end{figure}

On Figure~\ref{fig:hyster_1}, we can observe a typical pathwise hysteresis
phenomenon, which not only illustrates Proposition~\ref{prop:hyster} but also
suggests that the result of this Proposition could be well improved by proving a
almost sure lower bound (probably for sufficiently small values $\eta$). On
Figure~\ref{fig:hyster_1}, the forward path (red curve) is almost stuck to the
value $1$ on the interval $[0,1]$, we could then be tempted to think that the
lower bound of Proposition~\ref{prop:hyster} lacks some accuracy. 

On the contrary, when $\eta$ becomes small, which corresponds to a slower scale
for the variations of the external field, Figures~\ref{fig:hyster_2}
and~\ref{fig:zoom_hyster_2} show a very sharp revolving around $t=1/2$. The
evolution of $\E \lambda^\eta_t \cdot \bb$ is very steep in the neighborhood of
$1/2$, which corresponds to a change of sign of $b(t)$. By closely looking at
Figure~\ref{fig:zoom_hyster_2}, we notice that the lower bound $1/h^\eta(t)$ is
crossed for $t$ slightly larger than $1/2$. This phenomenon will be all the more
pronounced as $\eta$ goes to zero. In that sense, the lower bound $1/h^\eta(t)$
becomes nearly optimal for $t$ lower but close to $1/2$.

\section{Conclusion}

In this article, we analyzed the long time behaviour of the SDE modeling the
evolution of a magnet submitted to a perturbated external field.  The rate of
convergence of the magnetic moment is particularly interesting and holds for any
dissipation coefficient $\alpha >0$. This result has been obtained by combining
the ODE technique with Itô's formula. The second result concerns the hysteresis
behaviour of the system induced by the stochastic perturbation.  These two
results illustrate the dissipative effects of a stochastic perturbation on a
ferromagnet by giving a first glimpse on how thermal effects can be modeled in the
framework of micromagnetism.

\clearpage

\appendix

\begin{lemma}\label{lem:intexp}
  Let $a >0$ and $f$ be a continuous function such that $\lim_{t \rightarrow
  +\infty} h(t) f(t) = 0$. Then, $\displaystyle \expp{-a h(t)} \int_0^t f(u)
  \expp{a h(u)} du \xrightarrow[t \rightarrow +\infty]{} 0$.
\end{lemma}

\begin{proof}
  The function $h$ is a one to one map from $[0, +\infty)$ to $[h(0), +\infty)$.
  Hence, a change of variable yields
  \begin{align*}
    \int_0^t f(u) \expp{a h(u)} du &= \int_{h(0)}^{h(t)}
    \frac{f(h^{-1}(v))}{h'(h^{-1}(v))}  \expp{a v} dv.
  \end{align*}
  Remember that $h'(s) = \frac{\varepsilon^2 (\alpha^2+1)}{h(s)}$. Hence, as the
  function $h^{-1}$ is increasing, we deduce that
  $\frac{f(h^{-1}(v))}{h'(h^{-1}(v))} \xrightarrow[v \rightarrow
  +\infty]{} 0$. Then, it is easy to show the result.
\end{proof}

\begin{lemma}\label{lem:liminfexp}
  Let $a >0$ and $f$ be a continuous function. Then,
  $$
  \liminf_{t \rightarrow +\infty} \expp{-a t} \int_0^t f(u)
  \expp{a u} du \ge \frac{1}{a} \liminf_{t \rightarrow +\infty} f(t).
  $$
\end{lemma}

\begin{proof}
  We define $l = \liminf_{t \rightarrow +\infty} f(t)$. Let $\eta >0$, there
  exists $T>0$, such that for all $t \ge T$, $f(t) \ge l - \eta$.
  \begin{align*}
    \expp{-a t} \int_0^t f(u) \expp{a u} du & =  \expp{-a t} \int_0^T f(u) \expp{a
    u} du  +  \expp{-a t} \int_T^t f(u) \expp{a u} du  \\
    & \ge  \expp{-a t} \int_0^T f(u) \expp{a
    u} du  +  \expp{-a t} \int_T^t (l - \eta) \expp{a u} du  \\
    \liminf_{t \rightarrow +\infty} \expp{-a t} \int_0^t f(u) \expp{a u} du &
    \le \frac{l-\eta}{a}.
  \end{align*}
  As the inequality holds for all $\eta$, the result easily follows.
\end{proof}

\bibliographystyle{abbrvnat}
\bibliography{one-particule}

\end{document}